\documentclass[12pt,a4paper,reqno]{amsart}
\usepackage[margin=2.2cm]{geometry}
\usepackage{microtype}
\usepackage[cal=euler]{mathalfa}
\usepackage{enumitem}
\setlist{label=(\arabic*)}

\input{commands_new}
\defcite\alp{alperin_unipotent}
\defcite\evs{evseev_conjugacy}
\defcite\assem{assem_elements}
\defcite\vis{vistoli_grothendieck}
\defcite\mac{macdonald_symmetric}

\opra\lrg{\Ring_\la}
\def\rbra#1{\bigg(#1\bigg)}
\def\bd{\mathbf d}
\def\cc#1#2{\mat{\ga(#1,#2)}} 
\def\vl#1{\sbr{#1}}
\def\Z{Z}
\def\z{s}
\def\ai{\mathrm{a.i.}}
\def\fs{\mathsf k}
\def\fm{\mathsf m}
\def\inj{\mathrm{inj}}

\def\isoto{\xto{\,\smash{ 
	\raisebox{-.4ex}{\ensuremath{\scriptstyle\sim}}
}\,}}

\usepackage{tikz-cd}
\usepackage[colorlinks=true,linkcolor=blue,citecolor=blue,urlcolor=blue]{hyperref}

\begin{document}
\title{Commuting matrices and volumes of linear stacks}
\author{Sergey Mozgovoy}

\email{mozgovoy@maths.tcd.ie}
\address{School of Mathematics, Trinity College Dublin, Ireland\newline\indent
Hamilton Mathematics Institute, Ireland}

\begin{abstract}
A conjecture by Higman asserts that the number of conjugacy classes in the unipotent group of upper triangular matrices over a finite field depends polynomially on the number of elements of the field. We will study several alternative counting problems arising from quiver representations and prove explicit formulas relating the corresponding invariants to the invariants of Higman's conjecture. To do this, we develop a general framework of linear stacks over small \'etale sites and study volumes of these stacks and of their substacks of absolutely indecomposable objects.
\end{abstract}

\maketitle

\section{Introduction}
Let $B_n\sbs\GL_n(\bF_q)$ be the subgroup consisting of all invertible upper-triangular matrices over a finite field $\bF_q$ and let $U_n\sbs B_n$ be the subgroup consisting of unipotent matrices.
Let $\cc GX$ denote the number of orbits of an action by a group $G$ on a finite set $X$.
The classical conjecture of Higman \cite{higman_enumerating} asserts that the number $\ga(U_n,U_n)$ of $U_n$-conjugacy classes in $U_n$ is a polynomial in $q$.
This conjecture was studied in 
\cite{
vera-lopez_conjugacya,
vera-lopez_some,
vera-lopez_conjugacy,
murray_conjugacy,
alperin_unipotent,
evseev_groups,
evseev_conjugacy,
goodwin_counting}.
Its refinement related to characters of groups associated to finite algebras was studied in
\cite{lehrer_discrete,
isaacs_characters,
isaacs_counting,
robinson_counting,
evseev_reduction}.
There is also some evidence that this conjecture could be false \cite{halasi_number,pak_higmans}.

\medskip
Alternatively, consider the number $\ga(B_n,U_n)$ of $B_n$-conjugacy classes in $U_n$ \cite{hille_classification,evseev_groups,evseev_conjugacy}.
It was proved by Evseev \cite{evseev_groups} that the numbers $\cc{U_n}{U_n}$ are polynomials in $q$ if and only if $\cc{B_n}{U_n}$ are.
We will see later that the latter numbers have a better behavior than the former.
By Burnside's lemma we have
\begin{equation}
\cc{U_n}{U_n}=\frac{\#{\sets{(x,y)\in U_n\xx U_n}{xy=yx}}}
{\#{U_n}},
\end{equation}
hence Higman's conjecture is equivalent to the statement that the variety of commuting matrices in $U_n$
(or of commuting strictly upper-triangular matrices) is polynomial-count (see~\S\ref{volume}).
The algebraic variety of commuting matrices in $U_n$ is merely an intersection of several quadrics in an affine space of dimension $n^2-n$.
The difficulty of its point counting is in stark contrast to the simplicity of point counting of the variety of arbitrary commuting matrices.
Let $M_n(\bF_q)$ be the algebra of order $n$ square matrices over $\bF_q$ and let $C_n\sbs M_n(\bF_q)\xx M_n(\bF_q)$ be the subvariety of commuting matrices.
Then $C_n$ is polynomial-count by the classical result of Feit and Fine \cite{feit_pairs}
\begin{equation}
\label{ff1}
\sum_{n\ge0}\frac{\# C_n}{\#\GL_n(\bF_q)}t^n
=\prod_{i\ge1}\prod_{j\ge0}\frac1{1-q^{1-j}t^i}.
\end{equation}

\medskip
More generally, assume that we have a family of finite-dimensional algebras
$A=A(\bk)$, for finite fields $\bk=\bF_q$ (the earlier example corresponds to the algebra of upper-triangular matrices).
Let $A^0\sbs A$ be the set of nilpotent elements,
$A^*\sbs A$ be the set of invertible elements, and let $A^a=A$ for notation reasons.
Equivalently, these are the subsets $A^Z\sbs A$ consisting of elements with the minimal polynomial having roots in $Z(\ubar\bk)$, where $Z\sbs\bA^1$ equals respectively
\begin{equation}
\label{eq:Zs}
Z_0=\set0,\qquad Z_*=\bA^1\ms\set0,\qquad Z_a=\bA^1.
\end{equation}
We can ask if the (families of) numbers of commuting pairs
\begin{equation}
\label{eq:com var}
C^{\z_1,\z_2}=\sets{(x,y)\in A^{\z_1}\xx A^{\z_2}}{xy=yx},\qquad \z_i\in\set{0,*,a},
\end{equation}
are polynomial-count, and if there is a relation between them.
For example, given a finite poset~$S$,
consider its incidence algebra $A=\bk S$ -- having the basis $e_{xy}$, for $x\le y$, and multiplication $e_{xy}e_{y'z}=\de_{yy'}e_{xz}$.
One knows that there exist posets $S$ such that the number of commuting pairs in $A^0\xx A^0$ is not polynomial-count \cite{halasi_number,pak_higmans}.
These counterexamples provide the main indication that Higman's conjecture could be false.

\medskip
In this paper we will study the following class of algebras.
Let $Q$ be an acyclic quiver and let $A=\bk Q$ be its path algebra over $\bk=\bF_q$.
For every vertex $i\in Q_0$, 
let $P(i)$ be the corresponding indecomposable projective representation.
Given $\bd=(d_i)_{i\in Q_0}\in\bN^{Q_0}$,
define $P_\bd=\bop_i P(i)^{d_i}$ and $A_\bd=\End(P_\bd)$.
Note that $A_\bd=A$ for $\bd=(1,\dots,1)$.
For example, consider the quiver
$$1\to 2\to \dots\to n$$
Then $P_\bd$ is a representation of the form
$$\bk^{d_1}\emb\bk^{d_1+d_2}\emb\dots\emb\bk^{m},\qquad m=d_1+\dots+d_n.$$
Therefore $A_\bd=\End(P_\bd)$ can be identified with the algebra of block upper-triangular matrices in $M_m(\bk)$, with blocks of size $d_i$ on the diagonal.
The analogue of Higman's conjecture for these algebras was studied by Evseev
\cite{evseev_groups,evseev_conjugacy}.
In particular, for $\bd=(1,\dots,1)$, the algebra $A_\bd=A$ can be identified with the algebra of upper-triangular matrices and we have $B_n=A_\bd^*$, $U_n=1+A_\bd^0$.
We can generalize Higman's conjecture (\cf \cite{evseev_conjugacy}):

\begin{conjecture}
\label{conj1}
For any acyclic quiver $Q$ and any $\bd\in\bN^{Q_0}$,
the number of commuting pairs in $A_\bd^0\xx A_\bd^0$ is polynomial-count.
\end{conjecture}

Let us now discuss a relation between the above numbers and other counting problems.
Even though we don't know if the above numbers are polynomial-count, we can perform with them various operations (including plethystic ones \S\ref{lambda}) using a technical tool called the volume ring (see \S\ref{volume} or \cite{mozgovoy_poincare}).
This is a \la-ring $\cV=\prod_{n\ge1}\bQ$ whose elements
comprise the counting information of objects defined over different finite fields.
Given a family of rational numbers $c(K)$, for finite field extensions $K/\bk=\bF_q$, define its volume $[c]=(c(\bF_{q^n}))_{n\ge1}\in\cV$.
Similarly, given a family
of finite sets $X(K)$, for finite field extensions $K/\bk$ (for example, an algebraic variety $X$ over \bk), define its volume 
\begin{equation}
[X]=\rbr{\#X(\bF_{q^n})}_{n\ge1}\in\cV.
\end{equation}
Such family is called polynomial-count if there exists a polynomial $f\in\bQ[x]$ such that $[X]=[f]=(f(q^n))_{n\ge1}$.
Let us denote $[\bA^1]=(q^n)_{n\ge1}$ simply by $q$, hoping it will not cause any confusion.

\smallskip
For any $\z_1,\z_2\in\set{0,*,a}$, define $C^{\z_1,\z_2}_\bd\sbs A_\bd^{s_1}\xx A_\bd^{s_2}$ as in~\eqref{eq:com var}
and let
\begin{equation}
\label{eq:Hss1}
\sH^{\z_1,\z_2}(t)=\sum_{\bd\in\bN^{Q_0}}\frac{\vl{ C^{\z_1,\z_2}_\bd}}{\vl{\Aut(P_\bd)}}t^\bd\in\cV\pser{t_i\rcol i\in Q_0}.
\end{equation}
We are going to prove relations between the series $\sH^{\z_1,\z_2}(t)$ for various $\z_1,\z_2$.
To get a taste of things to come, it is instructive to consider representations of the trivial quiver (they are just vector spaces) equipped with one endomorphism.
Thus, we consider
\begin{equation}
\sH^\z(t)=\sum_{n\ge0}\frac{\vl{M_n^\z}}{\vl{\GL_n}}t^{n},\qquad \z\in\set{0,*,a}.
\end{equation}
Then
\begin{equation}
\sH^*(t)
=\frac1{1-t},\qquad
\sH^a(t)=\sum_{n\ge0}\frac{t^n}{(q\inv)_n}
=\prod_{i\ge0}\frac1{1-q^{-i}t},
\end{equation}
where
$(q)_n=\prod_{i=1}^n(1-q^i)$ and the last equation follows from the $q$-binomial theorem.
Let us rewrite these equations using the plethystic exponential (see \S\ref{lambda}) on $\cV\pser{t}$, satisfying
\begin{equation}
\Exp(f+g)=\Exp(f)\Exp(g),\qquad 
\Exp(q^kt^n)=\frac1{1-q^kt^n}.
\end{equation}
We obtain
$$\sH^*(t)=\Exp(t),\qquad \sH^a(t)=\Exp\rbr{\frac q{q-1}t},$$
and it is quite natural to guess that the numerator and the denominator of the fraction correspond to the number of elements in $\bF_q^a$ and $\bF_q^*$ respectively, and moreover, that 
\begin{equation}
\sH^0(t)=\Exp\rbr{\frac t{q-1}}
=\prod_{i\ge0}\frac1{1-q^{-i-1}t}=\sum_{n\ge0}\frac{q^{-n}t^n}{(q\inv)_n}.
\end{equation}
The last formula is equivalent to the equation $\#M^0_n(\bF_q)=q^{n^2-n}$ which is a classical result of Fine and Herstein \cite{fine_probability}.
One more ingredient of the above formulas, not easy to detect with a naked eye, is that the series $\sA(t)=t$ under the plethystic exponential in $\sH^*(t)$ counts (absolutely) indecomposable vector spaces up to an isomorphism.
An objects is called absolutely indecomposable
if it stays indecomposable after field extensions.
Now we are ready to formulate:

\begin{theorem}
\label{th:main1}
For any $\z_1,\z_2\in\set{0,*,a}$, we have
\begin{equation}
\sH^{\z_1,\z_2}(t)=\Exp\rbr{\frac{[\Z_{\z_1}][\Z_{\z_2}]}{q-1}\sA(t)},
\end{equation}
where $\sA(t)=\sum_\bd[\sA_\bd] t^\bd$, $[\sA_\bd]$ counts isomorphism classes of absolutely indecomposable objects of the form $(P_\bd,\vi)$ with nilpotent
$\vi\in\End(P_\bd)$,
and $Z_s$ are defined in \eqref{eq:Zs}.
In particular,
Conjecture~\ref{conj1} is equivalent to
\begin{enumerate}
\item $C^{\z_1,\z_2}_\bd$ are polynomial-count for all $\bd\in\bN^{Q_0}$.
\item $\sA_\bd$ are polynomial-count for all $\bd\in\bN^{Q_0}$.
\end{enumerate}
\end{theorem}

For example, for the trivial quiver, we obtain
\begin{equation}
\sH^{a,a}(t)=\Exp\rbr{\frac{q^2}{q-1}\sum_{n\ge1}t^n}
\end{equation}
which is exactly the Feit-Fine formula \eqref{ff1}.
Invariants $[\sA_\bd]$ are analogues of the Kac polynomials \cite{kac_infinite,kac_root} for quiver representations.
Therefore the following conjecture (satisfied in all known examples) is not surprising.

\begin{conjecture}
For any acyclic quiver $Q$ and any $\bd\in\bN^{Q_0}$,
there exists $f_\bd\in\bN[x]$ such that $[\sA_\bd]=[f_\bd]$.
\end{conjecture}

In order to prove the above theorem, we will count objects in the general context of linear stacks (\cf the notion of a sheaf-category in  \cite{dobrovolska_moduli}).
Assume that, for any finite field extension $K/\bk=\bF_q$, we have an additive $K$-linear category $\cA(K)$ with finite-dimensional $\Hom$-spaces and splitting idempotents.
Assume that for finite fields extensions $L/K/\bk$, we have compatible $K$-linear (restriction) functors $\rho_{L/K}:\cA(K)\to\cA(L)$.
For example, let $Q$ be a quiver and let $\cA(K)=\Rep(Q,K)$ be the category of quiver representations over $K$, with the restriction functors $\rho_{L/K}:M\mto M\ts_KL$.
As a different example, let $X$ be an algebraic variety over \bk and let $\cA(K)=\Coh X_K$, where $X_K=X\xx_{\Spec\bk}\Spec K$, with the restriction functor $\rho_{L/K}:F\mto F\ts_KL$.
We will assume that the above data satisfies certain descent conditions, meaning that it defines a stack over the small \'etale site $\Spec(\bk)_\et$ \S\ref{linear stacks}.

\medskip
Now let us count objects of the stack \cA.
Assume that there is a lattice $\Ga\iso\bZ^n$ and a group homomorphism $\cls:K_0(\cA(K))\to\Ga$, for every finite field extension $K/\bk$, compatible with the restriction functors and such that, for every  $\bd\in\Ga$, there exist only finitely many isomorphism classes of objects $X\in\cA(K)$ with $\cls(X)=\bd$.
Given a subvariety $\Z\sbs\bA^1$
and a finite-dimensional $K$-algebra $A$, define $A^\Z\sbs A$
as before.
Define
\begin{equation}
\label{eq:HZ1}
\sH^\Z_\cA(t)=\sum_{\bd\in\Ga}[H_d^\Z]t^d,\qquad
\sH^\Z_\bd(K)=\sum_{\ov{X\in\cA(K)\qt\sim}{\cls X=\bd}}
\frac{\#\End(X)^\Z}{\#\Aut(X)}.
\end{equation}
One can interpret this series as a series of volumes of a new stack $\cA^Z$ consisting of pairs $(X,\vi)$, where $X$ is an object of \cA and $\vi\in\End(X)^Z$ \S\ref{sec:counting endo}.
Let $\sA_\bd(K)$ be the number of isomorphism classes of absolutely indecomposable objects $X\in\cA(K)$ with $\cls(X)=\bd$ and let $\sA_\cA(t)=\sum_{\bd\in\Ga}[\sA_\bd]t^\bd$ be the corresponding series of volumes.

\begin{theorem}
\label{th:main2}
We have
\begin{equation}
\sH^\Z_\cA(t)=\Exp\rbr{\frac{[\Z]}{q-1}\sA_\cA(t)}.
\end{equation}
\end{theorem}

The above statement in the case of quiver representations and $\Z=\bA^1\ms\set0$ or $\Z=\set0$ goes back to Hua \cite{hua_counting} (see \cite{mozgovoy_computational} for the plethystic formulation of these results).
Other cases of the above result can be found in
\cite{mozgovoy_poincare,mozgovoy_motivicb,
schiffmann_indecomposable,mozgovoy_countinga,bozec_number}.


\medskip
The paper is organized as follows.
In \S\ref{prelim} we recall preliminary material on \la-rings and plethystic operations, the volume ring, Krull-Schmidt categories and the radical of a category. 
In \S\ref{linear stacks} we introduce linear stacks and study descent properties of indecomposable and absolutely indecomposable objects of such stacks.
In \S\ref{sec:volumes of stacks} we prove our main result on the volumes of linear stacks (\cf Theorem \ref{th:main2}).
In \S\ref{sec:comm1} we apply this result to the case of commuting varieties and prove Theorem \ref{th:main1}.
We discuss possible approaches to Higman's conjecture, based on the developed techniques, in Remarks \ref{rem: s method} and \ref{rem:evs result}.
In \S\ref{sec:examples} we provide some examples of the computation of invariants $[\sA_\bd]$ introduced in Theorem \ref{th:main1}.

\medskip
I would like to thank Markus Reineke for pointing my attention to Higman's conjecture and to thank Francis Brown for several useful discussions.
I learned with great sadness about the death of 
Anton Evseev who did so many contributions to questions related to Higman's conjecture.

\section{Preliminaries}
\label{prelim}

\subsection{\tpdf{\la}{Lambda}-rings and plethystic operations}
\label{lambda}
For simplicity we will introduce only \la-rings without \bZ-torsion.
To make things even simpler we can assume that our rings are algebras over~\bQ.
The reason is that in this case the axioms of a \la-ring can be formulated in terms of Adams operations.
For more details on \la-rings see \cite{getzler_mixed,mozgovoy_computational}.
Define the graded ring of symmetric polynomials
\begin{equation}
\La_n=\bZ[x_1,\dots,x_n]^{S_n},
\end{equation}
where $\deg x_i=1$.
Define the ring of symmetric functions
$\La=\ilim\La_n,$
where the limit is taken in the category of graded rings.
For any commutative ring $R$, define $\La_R=\La\ts_\bZ R$.
As in \mac, define generators of $\La$ (complete symmetric and elementary symmetric functions)
$$
h_n=\sum_{i_1\le\dots\le i_n}x_{i_1}\dots x_{i_n},\qquad
e_n=\sum_{i_1<\dots<i_n}x_{i_1}\dots x_{i_n},\qquad
$$
and generators of $\La_\bQ$ (power sums)
$$p_n=\sum_{i}x_i^n.$$
The elements $h_n,e_n,p_n$ have degree $n$.
We also define $h_0=e_0=p_0=1$ for convenience.

A \la-ring $R$ is a commutative ring equipped with a pairing, called plethysm,
$$\La\xx R\to R,\qquad (f,a)\mto f\circ a=f[a]$$
such that with $\psi_n=p_n[-]:R\to R$, called Adams operations, we have
\begin{enumerate}
\item The map $\La\to R,\, f\mto f[a]$, is a ring homomorphism, for all $a\in R$.
\item $\psi_1:R\to R$ is an identity map.
\item The map $\psi_n:R\to R$ is a ring homomorphism,
for all $n\ge1$.
\item $\psi_m\psi_n=\psi_{mn}$, for all $m,n\ge1$.
\end{enumerate}

\begin{remark}\br
\begin{enumerate}
\item
The first axiom implies that it is enough to specify just Adams operations $\psi_n$ or \si-operations $\si_n=h_n[-]$ or \la-operations $\la_n=e_n[-]$.
It also implies that $1[a]=1$, for all $a\in R$.

\item
We equip algebras of the form 
$\bQ[x_1,\dots,x_k]$, $\bQ\pser{x_1,\dots,x_k}$
with the \la-ring structure
$$\psi_n(f(x_1,\dots,x_k))=f(x_1^n,\dots,x_k^n).$$

\item
Similarly, given a \la-ring $R$, we equip algebras $R[t]$ and $R\pser t$ with the \la-ring structure
$$\psi_n\rbra{\sum_{i\ge0} a_it^i}=\sum_{i\ge0}\psi_n(a_i)t^{in}.$$
\item
The ring \La can be itself equipped with the \la-ring structure using the same formula
$$
\psi_m(f)=f(x_1^m,x_2^m,\dots),\qquad f\in\La.
$$
In particular $p_m[p_n]=p_{mn}$.
\item
If $R$ is a \la-ring, then $f\circ(g\circ a)=(f\circ g)\circ a$ for all $f,g\in\La$ and $a\in R$.
\end{enumerate}
\end{remark}


Define a filtered \la-ring $R$ to be a \la-ring equipped with a filtration by ideals
$$R=F^0R\sps F^1R\sps\dots$$
such that $F^iR\cdot F^jR\sbs F^{i+j}R$ and $\psi_n(F^iR)\sbs F^{ni}R$.
It is called complete if the natural homomorphism $R\to\ilim R/F^iR$ is an isomorphism.
Given a complete \la-ring $R$,
define the (usual) exponential and logarithm
\begin{equation}
\exp:F^1R\to 1+F^1R,\qquad a\mto\sum_{n\ge0}\frac {a^n}{n!},
\end{equation}
\begin{equation}
\log:1+F^1R\to F^1R,\qquad 1+a\mto\sum_{n\ge1}(-1)^{n-1}\frac {a^n}{n}.
\end{equation}
Define the exponentiation operation
\begin{equation}
(1+F^1R)\xx R\mto 1+F^1R,\qquad (a,b)\mto a^b=\exp(b\log(a)).
\end{equation}
Define the plethystic exponential
\begin{equation}
\Exp:F^1R\to 1+F^1R,\qquad a\mto\sum_{n\ge0}h_n[a]=\exp\rbr{\sum_{n\ge1}\frac{\psi_n(a)}n}.
\end{equation}
It satisfies
\begin{equation}
\Exp(a+b)=\Exp(a)\Exp(b).
\end{equation}
If $t\in F^1R$ satisfies $\psi_n(t)=t^n$, for all $n\ge1$ (we say that $t$ is linear in this case), then
\begin{equation}
\Exp(t)=\sum_{n\ge0}t^n=\frac1{1-t},
\end{equation}
where the last expression means the exponentiation $(1-t)\inv$.
Define the plethystic logarithm to be the inverse of \Exp
\begin{equation}
\Log:1+F^1R\to F^1R,\qquad
1+a\mto\sum_{n\ge1}\frac{\mu(n)}{n}\psi_n(\log(1+a)).
\end{equation}
Define the plethystic exponentiation (or power structure)
\begin{equation}
(1+F^1R)\xx R\mto 1+F^1R,\qquad (a,b)\mto \Pow(a,b)=\Exp(b\Log(a)).
\end{equation}
The following result was proved in \cite[Appendix]{mozgovoy_computational}:

\begin{theorem}
\label{power}
Let $R$ be a complete \la-ring and $a\in 1+F^1R$, $b\in R$.
Let $b_n\in R$  satisfy $\psi_n(b)=\sum_{r\mid n}rb_r$,
for all $n\ge1$.
Then
\begin{equation}
\Pow(a,b)=\prod_{r\ge1}\psi_r(a)^{b_r}.
\end{equation}
\end{theorem}

Let us formulate the $q$-binomial theorem in terms of plethystic operations.
Recall that one defines the $q$-Pochhammer symbol $(a;q)_n=\prod_{i=0}^{n-1}(1-aq^i)$, for $n\ge0$ (or $n=\infty$).
Then the $q$-binomial theorem states
\begin{equation}
\sum_{n\ge0}\frac{(a;q)_n}{(q;q)_n}t^n
=\frac{(at;q)_\infty}{(t;q)_\infty}.
\end{equation}
If $a,t,q\in R$ are linear and $t\in F^1R$,
we can write the above equation as
\begin{equation}
\sum_{n\ge0}\frac{(a;q)_n}{(q;q)_n}t^n
=\Exp\rbr{\frac{1-a}{1-q}t}.
\end{equation}

\subsection{Volume ring}
\label{volume}
Following \cite{mozgovoy_poincare},
we will introduce 
in this section a \la-ring which is an analogue of the Grothendieck ring of algebraic varieties or the ring of motives.
We define it to be the ring $\cV=\prod_{n\ge1}\bQ$ with Adams operations
\begin{equation}
\psi_m(a)=(a_{mn})_{n\ge1},\qquad a=(a_n)_{n\ge1}\in \cV,
\end{equation}
and call it the volume ring or the ring of counting sequences \cite{mozgovoy_poincare}. 

Let us fix a finite field $\bk=\bF_q$.
Given an algebraic variety $X$ over $\bk$
(or given a family of finite sets $X(K)$, for finite field extensions $K/\bk$),
define its volume
\begin{equation}
[X]=(\# X(\bF_{q^n}))_{n\ge1}\in \cV.
\end{equation}
More generally, given a finite type algebraic stack $\cX$ over $\bk$, define its volume
\begin{equation}
[\cX]=(\# \cX(\bF_{q^n}))_{n\ge1}\in \cV,
\end{equation}
where we define, for a finite groupoid $\cG=\cX(\bF_{q^n})$,
\begin{equation}
\#\cG=\sum_{x\in\cG\qt\sim}\frac1{\#\Aut(x)}.
\end{equation}

Consider the algebra $\bQ[\bq]$, where \bq is a variable, equipped with the usual \la-ring structure $\psi_n(f(\bq))=f(\bq^n)$.
There is an injective \la-ring homomorphism
$$\bQ[\bq]\to\cV,\qquad  f\mto [f]=(f(q^n))_{n\ge1}.$$
An element $a\in\cV$ is called polynomial-count if it is contained in the image of $\bQ[\bq]$.
We will usually identify $\bQ[\bq]$ with its image in $\cV$.
The element $[\bA^1]=(q^n)_{n\ge1}=\bq$ is called the Lefschetz volume.
In what follows we will write $q$ instead of $\bq$, hoping it will not lead to any confusion.

\subsection{Radical of a category}
\label{sec:rad}

Let $A$ be a ring (or an algebra over a field $K$).
Its (Jacobson) radical $\rad A$ is defined to be the intersection of maximal left ideals of $A$.

\begin{theorem}
For any ring $A$, we have
\begin{enumerate}
\item  $\rad A$ is a two-sided ideal.
\item $\rad A=\sets{a\in A}{1-ab\text{ is invertible }\forall b\in A}$.
\item $\rad A=\sets{a\in A}{1-ba\text{ is invertible }\forall b\in A}$.
\end{enumerate}
\end{theorem}

\begin{theorem}
For any finite-dimensional algebra $A$ over a field $K$, we have
\begin{enumerate}
\item $\rad A$ is nilpotent \assem[1.2.3].
\item An element $x\in A$ is invertible \iff its image in $A\qt\rad A$ is invertible.
\item Idempotents of $A\qt\rad A$ can be lifted to $A$ \assem[1.4.4].
\item If $L/K$ is a separable field extension, then $\rad (A\ts_KL)\iso(\rad A)\ts_KL$
\cite[Prop.~7.2.3]{bourbaki_algebre}.
\end{enumerate}
\end{theorem}

A ring $A$ is called local if it has a unique maximal left (or right) ideal.

\begin{theorem}[]
For a finite-dimensional algebra $A$, the following are equivalent \assem[1.4.6]
\begin{enumerate}
\item $A$ is local.
\item The only idempotents of $A$ are $0,1$.
\item $A\qt\rad A$ is a division algebra.
\end{enumerate}
\end{theorem}

Given a commutative ring $R$, define an $R$-linear category \cA to be a category enriched over $\Mod R$.
This means that the \Hom-sets in \cA are equipped with an $R$-module structure and the composition maps are $R$-linear.
Let $\cA$ be an additive category.
If $X\in\cA$ has a local endomorphism ring, then $X$ is indecomposable.
An additive category \cA is called Krull-Schmidt if every
object decomposes as a finite direct sum of indecomposable objects and if indecomposable objects have local endomorphism rings.
In this case
\begin{enumerate}
\item \cA is Karoubian (has splitting idempotents), that is, for every idempotent $e\in\End(X)$, there exists a decomposition $X\iso X_1\oplus X_2$ such that $X\to X_1\to X$ is equal to $e$.
\item 
A decomposition of an object into a finite direct sum of indecomposable objects is unique up to a permutation of summands (Krull-Schmidt theorem).
\end{enumerate}

If $K$ is a field and \cA is an additive, $K$-linear, Karoubian category with finite-dimensional \Hom-spaces,
then \cA is a Krull-Schmidt category.
Indeed, we can split every object, as long as it has non-trivial idempotents, and this process will eventually stop because of finite-dimensionality of \Hom-spaces.
If $X$ is an indecomposable object, then $\End(X)$ has only idempotents $0,1$, hence it is a local algebra.

\medskip
Given a Krull-Schmidt category \cA, define its radical $\rad_\cA$ to be the class of morphisms
\begin{equation}
\rad_\cA(X,Y)=\sets{a\in\cA(X,Y)}{1_X-ba\text{ is invertible }\forall b\in\cA(Y,X)},\qquad X,Y\in\cA.
\end{equation}

\begin{theorem}[]
We have \assem[A.3]
\begin{enumerate}
\item $\rad_\cA$ is a two-sided ideal in \cA.
\item
$\rad_\cA(X,Y)=\sets{a\in\cA(X,Y)}{1_Y-ab\text{ is invertible }\forall b\in\cA(Y,X)}.$
\item If $X=\bop_{i=1}^m X_i$ and $Y=\bop_{j=1}^nY_j$, then $\rad_\cA(X,Y)\iso\bop_{i,j}\rad_\cA(X_i,Y_j)$.
\item For any $X\in\cA$, the radical $\rad_\cA(X,X)$ is equal to the radical of the ring $\End_\cA(X)$.
\item If $X,Y$ are indecomposable, then $\rad_\cA(X,Y)$ is the set of all non-isomorphisms in $\cA(X,Y)$.
In particular, if $X\not\iso Y$, then $\rad_\cA(X,Y)=\cA(X,Y)$.
\end{enumerate}
\end{theorem}

\section{Linear stacks}
\label{linear stacks}

\subsection{General linear stacks}
In this section a field extension will always mean a finite separable field extension, unless otherwise stated.
The goal of this section is to formalize the structures appearing in the following example.
Given a finite-dimensional algebra $A$ over a field $\bk$, let $\mmod A$ denote the category of finitely generated (left) $A$-modules.
For any field extension $K/\bk$, let $A_K=K\ts_\bk A$ and let $\cA(K)=\mmod A_K$.
Given field extensions $L/K/\bk$ (with an embedding $\si:K\emb L$), there is a functor
\begin{equation}
\rho_{L/K}=\rho_\si:\cA(K)\to\cA(L),\qquad X\mto X_L=L\ts_KX.
\end{equation}

This data defines what we will call a linear stack over the small \'etale site of $\bk$.
Here are some of the properties of $\cA$:
\begin{enumerate}
\item The category $\cA(K)$ is $K$-linear and the functor $\rho_{L/K}:\cA(K)\to\cA(L)$ is $K$-linear on the \Hom-spaces, for any field extension $L/K$.
\item If $X,Y\in\cA(K)$ and $L/K$ is a field extension, then
$$\Hom(X_L,Y_L)\iso\Hom(X,Y)\ts_KL.$$
In particular, if $L/K$ is a Galois extension with $G=\Gal(L/K)$, then 
$$\Hom(X_L,Y_L)^G\iso\Hom(X,Y).$$
\item
Consider a Galois extension $L/K$ with $G=\Gal(L/K)$ and let $X\in\cA(K)$ and $Y=L\ts_KX\in\cA(L)$.
Every $\si\in G$ induces a functor $\rho_\si:\cA(L)\to\cA(L)$ so that $\rho_{\si\ta}\iso\rho_\si\rho_\ta$, for $g,h\in G$.
We define $Y^\si=\rho_\si\inv(Y)\in\cA(L)$ which is given by the same group $Y$ with the new multiplication $a\circ_\si y=\si(a)y$, for $a\in A_L$ and $y\in Y$.
Define a map 
$$\vi_\si:Y\to Y^\si,\qquad L\ts_K X\ni a\ts x\mto g(a)\ts x,$$
which is an isomorphism in $\cA(L)$.
These maps satisfy $\vi_{\si\ta}=\rho_\ta\inv(\vi_\si)\circ\vi_\ta$.
One can show that the category of pairs $(Y,\vi)$ as above is equivalent to the category $\cA(K)$.
\end{enumerate}

Now let us formalize the above structures.
Let $(\cC,\cO)$ be a ringed site.
This means that \cC is a category equipped with a Grothendieck topology, and \cO is a sheaf of commutative rings over~\cC.
Recall that a sheaf of $\cO$-modules (or an \cO-module) is a sheaf $\cF$ over \cC such that, for every $U\in\Ob\cC$, the set $\cF(U)$ is equipped with an $\cO(U)$-module structure
and, for every morphism $U\to V$ in \cC, the restriction map $\rho:\cF(V)\to \cF(U)$ is $\cO(V)$-linear.
Similarly, we define an \cO-stack (or an \cO-linear stack) to be a stack \cA over \cC \vis[\S4.1.3] such that,
for every $U\in\Ob\cC$,
the category $\cA(U)$ is $\cO(U)$-linear
and such that,
for every morphism $U\to V$ in \cC,
the restriction functor $\rho:\cA(V)\to\cA(U)$ is $\cO(V)$-linear.

We will be interested in just one particular ringed site.
Let \bk be a field and let $\cC=\Spec(\bk)_\et$ be the small \'etale site of $\Spec(\bk)$ \vis[Ex.~2.29].
Its opposite category is the category of separable (commutative, locally of finite type) algebras over \bk, that is, products of finite separable field extensions of \bk.
A sheaf \cF over \cC satisfies $\cF(\bigsqcup_i U_i)=\prod_i\cF(U_i)$, hence it is determined by $\cF(K)=\cF(\Spec K)$ for finite separable field extensions $K/\bk$, together with compatible restriction maps $\cF(K)\to \cF(L)$, for finite separable field extensions $L/K/\bk$, satisfying the gluing axiom for Galois extensions.
The latter condition means that $\cF(K)\iso \cF(L)^G$ for 
a Galois extension $L/K$ with the Galois group $G$
\cite[Prop.~2.1.4]{milne_etale}.
Define \cO to be the structure sheaf on the site \cC, that is,
$\cO(K)=K$ for every field extension $K/\bk$.
Given an \cO-module \cF, the set $\cF(K)$ is equipped with a structure of a $K$-vector space.
Moreover, for field extensions $L/K/\bk$ (with an embedding $\si:K\emb L$), there is a $K$-linear restriction map
$\rho_{L/K}=\rho_\si:\cF(K)\to\cF(L).$

Given an \cO-linear stack $\cA$, the category $\cA(K)$ is $K$-linear and the restriction functor $\rho_{L/K}=\rho_\si:\cA(K)\to\cA(L)$ is also $K$-linear.
We will always assume that $\cA(K)$ are additive Karoubian with finite-dimensional \Hom-spaces
and call such stack a linear stack (\cf the notion of a sheaf-category over \bk studied in \cite{dobrovolska_moduli}).
This implies that the categories $\cA(K)$ are Krull-Schmidt
\S\ref{sec:rad}.
Given a field extension $L/K$,
we define $X_L=\rho_{L/K}(X)$, for $X\in\cA(K)$,
and we say that $Y\in\cA(L)$ descends to $X\in\cA(K)$ if $Y\iso X_L$.
We will say that an object $X\in\cA(K)$ is absolutely indecomposable if $X_L\in\cA(L)$ is indecomposable,
for every field extension $L/K$.

\begin{remark}
\label{rem:descent data}
Given a Galois extension $L/K$ with the Galois group $G=\Gal(L/K)$, there is an isomorphism
$$L\ts_KL\iso\prod_{\si\in G}L,\qquad a\ts b\mto(a\cdot\si b)_{\si\in G}.$$
Consider canonical homomorphisms
$$
i_1,i_2:L\to L\ts_K L\iso\prod_{\si\in G}L\qquad
i_1(a)\mto (a)_{\si\in G},\qquad i_2(b)\mto (\si b)_{\si\in G}.
$$
and the corresponding restriction functors
$$\rho_1,\rho_2:\cA(L)\to\cA(L\ts_KL)\iso\coprod_{\si\in G}\cA(L)$$
An object with descent data $(Y,\vi)$ \wrt $L/K$ \vis[\S4.1.2] 
consists of $Y\in\cA(L)$ and an isomorphism $\vi:\rho_2X\iso\rho_1X$.
This corresponds to isomorphisms $\vi_\si:\rho_\si(Y)\to Y$ or $\vi_\si:Y\to Y^\si=\rho_\si\inv(Y)$, for $\si\in G$.
These isomorphisms should satisfy certain compatibility condition over $L\ts_KL\ts_KL$.
By stack axioms, the category $\cA(K)$ is equivalent to the category of the above descent data, where $X\in\cA(K)$ is sent to $X_L\in\cA(L)$.
\end{remark}

\begin{lemma}
\label{lm:hom tensor}
Consider field extensions $L/K/\bk$. Then
\begin{enumerate}
\item For any \cO-module \cF, we have $\cF(K)\ts_K L\iso \cF(L)$. 
\item For any linear stack \cA and $X,Y\in\cA(K)$, we have
$\Hom(X_L,Y_L)\iso\Hom(X,Y)\ts_KL$.
\end{enumerate}
\end{lemma}
\begin{proof}
1. Considering the Galois closure $L'$ of $L/K$, we get Galois extensions $L'/L$ and $L'/K$ and if the statement will be proved for these extensions then it will follow for $L/K$. Therefore we can assume that $L/K$ is Galois.
By the sheaf axioms, we have $\cF(K)=\cF(L)^G$, where $G=\Gal(L/K)$.
For the vector space $V=\cF(L)$, the canonical map $V^G\ts_KL\to V$ is an isomorphism by the Galois descent.
2. We use the previous statement and the fact that the map $L/K\mto\Hom(X_L,Y_L)$ defines a sheaf.
\end{proof}

Let $\cA$ be an  linear stack and let $X,Y\in\cA(K)$.
We will denote $\Hom_{\cA(K)}(X,Y)$ by $\Hom(X,Y)$ and denote $\rad_{\cA(K)}(X,Y)$ by $\rad(X,Y)$.
The corresponding category is usually clear from the context.

\begin{lemma}
\label{lm:rad tensor}
Let \cA be a linear stack, $L/K/\bk$ be field extensions
and $X,Y\in\cA(K)$. Then
$$\rad(X_L,Y_L)\iso\rad(X,Y)\ts_KL.$$
\end{lemma}
\begin{proof}
As the radical is additive (see \S\ref{sec:rad}), it is enough
to consider $Z=X\oplus Y$, 
$A=\End(Z)$, $B=\End(Z_L)$
and to show that $\rad B\iso (\rad A)\ts_K L$.
We know that $B\iso A\ts_K L$ by Lemma \ref{lm:hom tensor}.
On the other hand we know that $\rad(A\ts_K L)\iso(\rad A)\ts_KL$ for separable field extensions~\S\ref{sec:rad}.
\end{proof}

\begin{lemma}
\label{lem:iso descent}
Let \cA be a linear stack, $L/K/\bk$ be field extensions
and $X,Y\in\cA(K)$.
If $X_L\iso Y_L$, then $X\iso Y$.
\end{lemma}
\begin{proof}
If $X,Y$ don't have isomorphic indecomposable summands, then we have $\rad(X,Y)=\Hom(X,Y)$ by \S\ref{sec:rad}.
But this would imply that $\rad(X_L,Y_L)=\Hom(X_L,Y_L)$, hence $X_L$, $Y_L$ don't have isomorphic indecomposable summands, a contradiction.
Choosing isomorphic indecomposable summands in $X,Y$, we can remove them and repeat the argument.
\end{proof}

\subsection{Linear stacks over finite fields}
From now on we assume that \bk is a finite field.
Then every finite field extension $K/\bk$ is automatically separable and Galois.
If $L/K$ is a field extension and $Y\iso X_L$ for some $X\in\cA(K)$, then $Y^\si\iso Y$ for all $\si\in\Gal(L/K)$.
The converse is also true, even without a requirement of descent data on $Y$.
The following series of results is analogous to \cite[\S2.4]{mozgovoy_poincare}.

\begin{lemma}
\label{lm:descent}
Let $G=\Gal(L/K)$ and $Y\in\cA(L)$ be such that $Y^\si\iso Y$ for all $\si\in G$.
Then there exists $X\in\cA(K)$ such that $Y\iso X_L$.
\end{lemma}
\begin{proof}
Let $n=[L\rcol K]$ and $\si\in G$ be the generator of the Galois group (which is cyclic).
Choose an isomorphism $f:Y\to Y^\si$ and let
$f^k:Y\to Y^{\si^k}$ be the corresponding composition.
In particular, we get an isomorphism $f^n:Y\to Y$.
The group $\Aut(Y)\sbs\Hom(Y,Y)$ is finite.
Choose $m\ge1$ such that $f^{mn}=1$ and choose a finite field extension $L'/L$ of degree $m$.
We have an epimorphism
$\Gal(L'/K)\to\Gal(L/K)$ and we can choose a generator $\si'\in\Gal(L'/K)$ that is mapped to~\si.
Now we define $\vi_{\si'}=\rho_{L'/L}(f)$ and extend it to the action of $\Gal(L'/K)$ on $Y_{L'}$ (meaning a descent data on $Y_{L'}$).
In this way we find $X\in\cA(K)$ such that $\rho_{L'/K}(X)\iso Y_{L'}=\rho_{L'/L}(Y)$.
But $\rho_{L'/K}(X)\iso\rho_{L'/L}(\rho_{L/K}(X))$,
hence $\rho_{L/K}(X)\iso Y$ by Lemma~\ref{lem:iso descent}.
\end{proof}

\begin{remark}
\label{rem:min field}
In view of the previous lemma, given a field extension $L/\bk$ and $Y\in\cA(L)$, define the minimal field of definition
$$\fm(Y)=K=L^G,\qquad G=\sets{\si\in\Gal(L/\bk)}{Y^\si\iso Y}.$$
Then $Y^\si\iso Y$ for all $\si\in G=\Gal(L/K)$, hence there exists $X\in\cA(K)$ such that $Y\iso X_L$.
If there exists a field extension $L/K'/\bk$ and $X'\in\cA(K')$ such that $Y\iso X'_L$, then $Y^\si\iso Y$ for all $\si\in\Gal(L/K')$, hence $\Gal(L/K')\sbs G$ and $K\sbs K'$.
\end{remark}

\begin{remark}
Let $K/\bk$ be a field extension and let $X\in\cA(K)$.
Define the splitting algebra
\begin{equation}
\fs(X)=\Hom(X,X)\qt\rad(X,X).
\end{equation}
This is a semi-simple finite dimensional algebra over $K$, hence it is isomorphic to a product of matrix rings over division algebras by Wedderburn-Artin theorem.
Every such division algebra is finite, hence it is actually a field extension of $K$ by Wedderburn's little theorem.
For any field extension $L/K$, we have by Lemmas \ref{lm:hom tensor}, \ref{lm:rad tensor}
\begin{equation}
\fs(X_L)\iso\fs(X)\ts_KL
\end{equation}
\end{remark}

\begin{lemma}
Let $K/\bk$ be a field extension and $X\in\cA(K)$. Then
\begin{enumerate}
\item $X$ is indecomposable if and only if $\fs(X)$ is a field.
\item
$X$ is absolutely indecomposable if and only if $\fs(X)=K$.
\end{enumerate}
\end{lemma}
\begin{proof}
1. $X$ is indecomposable if and only if $\End(X)$ is a local algebra if and only if $\fs(X)$ is a division algebra.
Finite division algebras are fields.

2. If $\fs(X)=K$, then for every field extension $L/K$ we have $\fs(X_L)\iso \fs(X)\ts_KL\iso L$, hence $X_L$ is indecomposable. We conclude that $X$ is absolutely indecomposable. If $X$ is indecomposable and $L=\fs(X)\ne K$, then $\fs(X_L)\iso L\ts_KL$ is a product of several copies of $L$. The corresponding idempotents can be lifted to $\End(X_L)$, hence $X_L$ is not indecomposable.
\end{proof}

\begin{theorem}
\label{th:splitting}
We have
\begin{enumerate}[ref={\thetheorem\,(\arabic*)}]
\item \label{th:splitting:1}
Let $X\in\cA(\bk)$ be indecomposable and let $K=\fs(X)$.
Then $X_K\iso\bop_{\si\in\Gal(K/\bk)}Y^\si$, for an absolutely indecomposable $Y\in\cA(K)$ with $\fm(Y)=K$.
\item  \label{th:splitting:2}
Let $Y\in\cA(K)$ be
indecomposable with $\fm(Y)=K$.
Then $\bop_{\si\in\Gal(K/\bk)}Y^\si\iso X_K$ 
for an indecomposable object $X\in\cA(\bk)$ with $\fs(X)=\fs(Y)$.
\end{enumerate}
\end{theorem}
\begin{proof}
1.
We have
$\fs(X_K)\iso\fs(X)\ts_\bk K\iso K\ts_\bk K\iso\prod_{\si\in G}K$, where $G=\Gal(K/\bk)$.
Lifting the corresponding idempotents, we obtain a decomposition $X_K\iso\bop_{\si\in G}Y_\si$, where $Y_\si\in\cA(K)$ are non-isomorphic objects
 (otherwise $\fs(X_K)$ would contain the ring of matrices of order $\ge2$)
with $\fs(Y^\si)=K$, hence they are absolutely indecomposable.
The action of $G$ on $X_K$ (more precisely, for every $\si\in G$, we have an isomorphism $X_K\to X_K^\si$)
induces the action on the set of isomorphism classes of $Y_\si$.
If this action is non-free, then we get a non-trivial decomposition $X_K=Y_1\oplus Y_2$ with both $Y_1,Y_2$ descending to $\bk$ by Lemma~\ref{lm:descent}.
This contradicts to the assumption that $X$ is indecomposable.
We can assume that $Y_\si\iso Y^\si$, where $Y=Y_1$.
Then $Y^\si\not\iso Y$, for $\si\ne1$, hence the minimal field of definition of $Y$ is $K$.

2.
Let $G=\Gal(K/\bk)$ and $X'=\bop_{\si\in G}Y^\si$.
Then $(X')^\si\iso X'$ for all $\si\in G$, hence by Lemma \ref{lm:descent} there exists $X\in\cA(\bk)$ such that $X_K\iso X'$ (we actually have a descent data on~$X'$, hence can obtain the object $X$ directly).
The objects $Y^\si$ are not isomorphic to each other as otherwise $Y$ would have a smaller field of definition.
Therefore $\fs(X')=\prod_{\si\in G}\fs(Y^\si)=\prod_{\si\in G} K$, hence $\fs(X)\iso \fs(X')^{G}\iso K$ and $X$ is indecomposable.
\end{proof}

\subsection{Scalar extension}
The previous results imply that a linear stack $\cA$ over a field $\bk$ is completely determined by the category $\cA(\bk)$.
Let us discuss this in more detail (\cf \cite{dobrovolska_moduli}).

Let $K$ be a field and $\cC$ be an additive $K$-linear category.
Given a finite field extension $L/K$, 
define the scalar extension $\cC_L$ of \cC \cite{artin_abstract,lowen_deformation,sosna_scalar,stalder_scalar,dobrovolska_moduli}
to be the category with objects $(X,\vi)$, where $X\in\cC$ and $\vi:L\to\End_\cC(X)$ is a homomorphism of $K$-algebras.
A morphism $f:(X,\vi)\to (Y,\vi')$ in $\cC_L$ is an element $f\in\Hom_\cC(X,Y)$ such that $\vi'(a)f=f\vi(a)$ for all $a\in L$.
We can interpret $\cC_L$ as the category of $K$-linear functors $\Fun_K(L,\cC)$, where $L$ is considered as the category with one object and the ring of endomorphisms $L$.
If \cC is Karoubian, then $\cC_L$ is also Karoubian.

The forgetful functor $\cC_L\to\cC$ has a left adjoint functor 
$(-)_L:\cC\to \cC_L$
constructed as follows.
Let $\mmod K$ denote the category of finite-dimensional vector spaces over~$K$.
Then, for every object $X\in\cC$, there exists a unique (up to a canonical natural isomorphism) $K$-linear functor~\cite{artin_abstract,lowen_deformation}
$$-\ts_K X:\mmod K\to\cC$$
such that $K\ts_KX=X$.
Given $V\in\mmod K$ and $X,Y\in\cC$, we have
$$\Hom_\cC(V\ts_K X,Y)\iso \Hom_K(V,\Hom_\cC(X,Y)).$$
Every element $a\in L$ induces a $K$-linear map $L\to L$, hence a morphism $\vi_X(a):L\ts_KX\to L\ts_KX$ in \cC.
In this way we obtain a functor
$$\rho_{L/K}=(-)_L:\cC\to\cC_L,\qquad
X\mto X_L=(L\ts_K X,\vi_X).$$

\begin{lemma}[see {\cite{artin_abstract,lowen_deformation,sosna_scalar}}]
The functor $(-)_L:\cC\to\cC_L$ is left adjoint to the forgetful functor $\cC_L\to\cC$.
\end{lemma}

\begin{lemma}
\label{lm:hom tensor1}
For any objects $X,Y\in\cC$, we have
$$\Hom_{\cC_L}(X_L,Y_L)\iso L\ts_K\Hom_\cC(X,Y).$$
\end{lemma}
\begin{proof}
By the previous result we have
$$\Hom_{\cC_L}(X_L,Y_L)\iso
\Hom_\cC(X,L\ts_K Y)\iso L\ts_K\Hom_\cC(X,Y).$$
\end{proof}

\begin{lemma}[\cf{} {\cite[Lemma 2.7]{sosna_scalar}}]
\label{lm:sosna galois}
Let \cC be an additive $K$-linear Karoubian category and
let $L/K$ be a finite Galois extension with the Galois group $G$.
Then $G$ acts on $\cC_L$ and there is an equivalence between the category \cC and the category of Galois descent data in $\cC_L$ (see Remark \ref{rem:descent data}).
\end{lemma}

On the other hand, let $\cC'_L$ 
be the category with the same objects as \cC and with morphisms 
$$\Hom_{\cC'_L}(X,Y)=L\ts_K\Hom_\cC(X,Y).$$
This category is not necessarily Karoubian even if \cC is.
We define $\hat\cC_L$ to be the Karoubian completion of $\cC'_L$.
It consists of objects $(X,e)$, where $X\in\cC'_L$ and $e\in\End_{\cC'_L}(X)$ is an idempotent.
Morphisms are defined by $$\Hom_{\hat\cC_L}((X,e),(Y,e'))=e'\Hom_{\cC'_L}(X,Y)e.$$
We identify $\cC_L'$ with the full subcategory of $\hat\cC_L$
by sending $X$ to $(X,1_X)$.
The category $\hat\cC_L$ has splitting idempotents.
For example, for any idempotent $e\in\End(X)$ we have
$$(X,1_X)\iso(X,e)\oplus(X,1_X-e).$$

\begin{lemma}
The functor $(-)_L:\cC\to\cC_L$ induces a full and faithful $L$-linear functor
$$(-)_L:\cC'_L\to\cC_L.$$
If \cC is Karoubian and $L/K$ is a finite Galois field extension, this functor extends to an equivalence of categories $\hat\cC_L\isoto\cC_L$.
\end{lemma}
\begin{proof}
The first statement follows from Lemma \ref{lm:hom tensor1}.
If $\cC$ is Karoubian, then $\cC_L$ is Karoubian and the functor $\cC_L'\to\cC_L$ extends to a full and faithful functor $\hat\cC_L\to\cC_L$.
Let $X=(X_0,\vi)\in\cC_L$.
Define $X^\si=\rho_\si\inv(X)=(X_0,\vi \si)\in\cC_L$ for $\si\in G=\Gal(L/K)$.
Then $\bop_{\si\in G}X^\si$ is naturally equipped with the Galois descent data, hence it is of the form $Y_L$ for some $Y\in\cC$ by Lemma~\ref{lm:sosna galois}.
This implies that $X$ is a direct summand of an object from $\cC'_L$.
\end{proof}

\begin{remark}
One can also prove the above equivalence 
of categories for a separable field extension $L/K$.
\end{remark}

\begin{lemma}
Let $\cA$ be a linear stack over a finite field $\bk$ and let $\cC=\cA(\bk)$.
Then, for any finite field extension $K/\bk$, we have $\cA(K)\iso\hat\cC_K\iso\cC_K$.
\end{lemma}
\begin{proof}
The second equivalence was proved in the previous lemma.
The restriction functor $\rho_{K/\bk}:\cC=\cA(\bk)\to\cA(K)$ induces a full and faithful functor $\cC_K'\to\cA(K)$ by Lemma
\ref{lm:hom tensor}.
It extends to a full and faithful functor $\hat\cC_K\to\cA(K)$ as $\cA(K)$ is Karoubian.
Given any object $X\in\cA(K)$, the object $\bop_{\si\in\Gal(K/\bk)}X^\si\in\cA(K)$ is naturally equipped with the Galois descent data, hence it is of the form $Y_K$ for some $Y\in\cA(\bk)$.
This implies that $X$ is a direct summand of an object from $\cC'_K$.
\end{proof}

\section{Volumes of linear stacks}
\label{sec:volumes of stacks}
Let \cA be a linear stack over a finite field $\bk=\bF_q$.
All field extensions are assumed to be finite.
Given an (essentially small) category \cC, let $\cC\qt\sim$ be the set of isomorphism classes of objects of~\cC.
As in the introduction, we assume that there is a lattice $\Ga\iso\bZ^n$ and group homomorphisms $\cls:K_0(\cA(K))\to\Ga$, for field extensions $K/\bk$, such that 
\begin{enumerate}
\item $\cls(X_L)=\cls(X)$, for field extensions $L/K/\bk$ and $X\in\cA(K)$.
\item
The set $\sets{X\in\cA(K)\qt\sim}{\cls(X)=\bd}$
is finite, for every  $\bd\in\Ga$.
\item
There exists a closed convex cone $C\sbs\Ga_\bR=\Ga\ts_\bZ\bR$ which is pointed ($C\cap(-C)=\set0$) and such that $\cls(X)\in C$ for all $X\in\cA(K)$.
This allows us to work with power series counting objects in \cA.
\end{enumerate}
Define the (weighted) volume series of \cA
\begin{equation}
\label{eq:vol series1}
\sH_\cA(t)=\sum_\bd[\sH_{\bd}]t^\bd,\qquad
\sH_{\bd}(K)
=\sum_{\ov{X\in\cA(K)\qt\sim}{\cls X=\bd}}\frac1{\#\Aut(X)}.
\end{equation}
Define also the unweighted volume series of \cA (\cf Remark \ref{rem:Hs})
\begin{equation}
\label{eq:unw vol series}
\sH^*_\cA(t)=\sum_\bd[\sH^*_{\bd}]t^\bd,\qquad
\sH^*_{\bd}(K)
=\#\sets{X\in\cA(K)\qt\sim}{\cls X=\bd}.
\end{equation}
In particular, consider the substack $\cA^\ai\sbs\cA$ of absolutely indecomposable objects of \cA and define the volume series of absolutely indecomposable objects of \cA
\begin{equation}
\label{eq:ai volume1}
\sA_\cA(t)=\sH^*_{\cA^\ai}(t)=\sum_\bd[\sA_\bd]t^\bd,
\end{equation}
where $\sA_\bd(K)$ denotes the number of isomorphism classes of absolutely indecomposable objects $X\in\cA(K)$ with $\cls X=\bd$, for field extensions $K/\bk$.
Note that we consider unweighted volumes here.
We will usually write $\sH(t)=\sH_\cA(t)$, $\sH^*(t)=\sH^*_\cA(t)$, $\sA(t)=\sA_\cA(t)$ if \cA is clear from the context.


\subsection{Counting indecomposable objects}
For every $\bd\in\Ga$ and $r\ge1$,
let $\sI_{\bd,r}(K)$ be the number of isomorphism classes of indecomposable objects $X\in\cA(K)$ with $\cls(X)=\bd$ and $\dim_K\fs(X)=r$.
Note that $\sA_\bd(K)=\sI_{\bd,1}(K)$ is the number of isomorphism classes of absolutely indecomposable objects.

\begin{proposition}
We have
\begin{equation}
\sA_\bd(\bF_{q^n})
=\sum_{r\mid n}r\sI_{r\bd,r}(\bF_q).
\end{equation}
\end{proposition}
\begin{proof}
Let $L=\bF_{q^n}$ and let $\ubar Y\in\cA(L)$ be an absolutely indecomposable object with $\cls\ubar Y=\bd$.
Let $K=\bF_{q^r}\sbs L$ be its minimal field of definition (hence $r\mid n$)
and let $Y\in\cA(K)$ be such that $\ubar Y\iso Y_L$ (see Remark \ref{rem:min field}).
This object is absolutely indecomposable and unique up to an isomorphism by Lemma \ref{lem:iso descent}.
The objects $Y^\si$, for $\si\in G=\Gal(K/\bk)$,
are not isomorphic to each other as otherwise 
$Y$ would descend to a smaller field by Lemma \ref{lm:descent}.
Applying Theorem~\ref{th:splitting:2},
we find an indecomposable object $X\in\cA(\bk)$ such that $X_K\iso\bop_{\si\in G}Y^\si$ and $\fs(X)=K$.
We have $\dim\fs(X)=[K:\bk]=r$ and
$$\cls(X)\iso\cls(X_K)=\sum_{\si\in G}\cls Y^\si=r\cls Y=r\bd.$$

Conversely, if $X\in\cA(\bk)$ is indecomposable and $K=\fs(X)$, then by Theorem~\ref{th:splitting:1}
$X\iso\bop_{\si\in G}Y^\si$,
for an absolutely indecomposable $Y\in\cA(K)$ with $\fm(Y)=K$.
Applying the restriction functor $\rho_{L/K}$ to the objects $Y^\si$, we obtain $r=[K:\bk]$ non-isomorphic, absolutely indecomposable objects in $\cA(L)$.
\end{proof}

\begin{remark}
The above proof implies that if $\sI_{\bd,r}(\bF_q)\ne0$, then $r\mid \bd$.
Using the M\"obius inversion formula, we obtain
\begin{equation}
\sI_{n\bd,n}(\bF_q)=\frac1n\sum_{r\mid n}\mu\rbr{\frac nr}\sA_\bd(\bF_{q^r}).
\end{equation}
\end{remark}

\begin{remark}
\label{rm:gauss}
The above formula,
in the case of quiver representations,
was proved by Kac \cite[\S1.14]{kac_root}.
But its first appearance can be attributed to Gauss who proved that the number $\vi_r(q)$ of degree $r$ monic irreducible polynomials over $\bF_q$ satisfies $q^n=\sum_{r\mid n}r\vi_r(q)$.
This equation follows from our result applied to the category of semisimple modules over $\bF_q[x]$.
The above proof follows the argument of \cite{mozgovoy_poincare}, where the case of stable vector bundles over a curve was considered.
\end{remark}

\begin{corollary}
\label{cor:abs vs ind}
We have 
\begin{equation}
\psi_n[\sA_\bd]=\sum_{r\mid n}r[\sI_{r\bd,r}],
\end{equation}
\begin{equation}
[\sI_{n\bd,n}]=\frac1n\sum_{r\mid n}\mu\rbr{\frac nr}\psi_r[\sA_\bd].
\end{equation}
\end{corollary}

\subsection{Counting objects with endomorphisms}
\label{sec:counting endo}
Let $\Z\sbs\bA^1_\bk$ be an algebraic subvariety.
For example, consider 
\begin{equation}
\label{eq:Zs2}
\Z_0=\set0,\qquad \Z_*=\bA_\bk^1\ms\set0,\qquad \Z_a=\bA_\bk^1.
\end{equation}
Given an algebra $A$ over a field $K/\bk$, define $A^\Z$ to be the set of algebraic elements $x\in A$ with the minimal polynomial having roots in $\Z(\ubar\bk)$ (we will say that such element $x$ and such polynomial have type $\Z$).

Define a new stack $\cA^\Z$ such that, for every field extension $K/\bk$, the category $\cA^\Z(K)$ has objects $X=(X_0,\vi)$, where $X_0\in\cA(K)$ and $\vi\in\End(X_0)^\Z$.
Morphisms $f:X\to Y$ between objects $X=(X_0,\vi)$ and $Y=(Y_0,\vi')$ are defined to be $f\in\Hom(X_0,Y_0)$ such that $\vi' f=f\vi$. 
Define $\cls X=\cls X_0\in\Ga$.
Consider the series of volumes of the new stack $\cA^\Z$ (\cf \eqref{eq:vol series1})
\begin{equation}
\label{eq:HZ2}
\sH^\Z_\cA(t)=\sH_{\cA^\Z}(t)
=\sum_{\bd\in\Ga}[\sH_\bd^\Z]t^\bd,\qquad
\sH_\bd^\Z(K)=\sum_{\ov{X\in\cA^\Z(K)\qt\sim}{\cls X=\bd}}\frac1{\#\Aut(X)}.
\end{equation}
We will usually write $\sH^\Z(t)=\sH^\Z_\cA(t)$ if \cA is clear from the context.
Note that
\begin{equation}
\sH^\Z_\bd(K)=\sum_{\ov{X_0\in\cA(K)\qt\sim}{\cls X_0=\bd}}
\frac{\#\End(X_0)^\Z}{\#\Aut(X_0)},
\end{equation}
hence the above definition of $\sH^\Z_\cA(t)$ coincides with \eqref{eq:HZ1}.

\begin{remark}
\label{rem:Hs}
Define $\sH^\z_\cA(t)=\sH^{\Z_\z}_\cA(t)$, for $\z\in\set{0,*,a}$.
Note that $\sH^*_\cA(t)$ thus defined coincides with the unweighted volume series of \cA \eqref{eq:unw vol series}.
\end{remark}

\begin{theorem}[\cf Theorem \ref{th:main2}]
\label{th:main proof}
We have
$$\sH^\Z_\cA(t)
=\Exp\rbr{\frac{[\Z]}{q-1}\cdot \sA_\cA(t)}.$$
\end{theorem}
\begin{proof}
For every field extension $K/\bk$ and every $\bd\in\Ga$, let $\Ind_\bd(K)$ be the set of isomorphism classes of indecomposable objects $X\in\cA(K)$ with $\cls(X)=\bd$.
Let $r_X=\dim_{K}\fs(X)$.
Every object $Y\in\cA(\bk)$ can be written in the form
$Y\iso \bop_{X\in\Ind(\bk)}X^{n_X}$, where $n:\Ind(\bk)\to\bN$ is a map with finite support.
We have
$$\fs(Y)=\End(Y)\qt\rad(Y,Y)
\iso\prod_X M_{n_X}(\fs(X))
.$$
An object $\vi\in\End(Y)$ is of type $\Z$ if and only if its image in $\fs(Y)$ is of type $\Z$ (as the radical is nilpotent) if and only if the corresponding components in $M_{n_X}(\fs(X))$ are of type~$\Z$.
The same argument can be applied also to invertible endomorphisms, corresponding to $\Z_*=\bA^1\ms\set0$.
We obtain
$$
\frac{\#\End(Y)^\Z}{\#\Aut(Y)}
=\frac{\#\fs(Y)^\Z}{\#\fs(Y)^{\Z_*}}
=\prod_X\frac{\n{M_{n_X}^\Z(\fs(X))}}{\n{\GL_{n_X}(\fs(X))}}.
$$
Let $\sH^\Z_{\Vect}(t)$ be the series of volumes of the stack \Vect of vector spaces (equipped with endomorphisms of type $Z$).
This means that
$$\sH_{\Vect}^\Z(t)=\sum_{n\ge0}[c_n]t^n\in\cV\pser t,\qquad
c_n(K)=\frac{\n{M_n^\Z(K)}}{\n{\GL_n(K)}}.
$$
Then
\begin{multline*}
\sum_{\bd\in\Ga}\sH^\Z_\bd(\bF_q)t^\bd=
\sum_{Y\in\cA(\bk)\qt\sim}\frac{\#\End(Y)^\Z}{\#\Aut(Y)}
t^{\cls Y}\\
=\sum_{n:\Ind(\bk)\to\bN}\prod_{X\in\Ind(\bk)}
c_{n_X}(\fs(X))t^{n_X\cls X}
=\prod_{X\in\Ind(\bk)}\rbra{\sum_{n\ge0}c_{n}(\fs(X))t^{n\cls X}}\\
=\prod_{\bd,r}\rbra{\sum_{n\ge0}
c_{n}(\bF_{q^r})t^{n\bd}}^{\sI_{\bd,r}(\bF_q)}
=\prod_{\bd,r}\rbra{\sum_{n\ge0}
c_{n}(\bF_{q^r})t^{nr\bd}}^{\sI_{r\bd,r}(\bF_q)}.
\end{multline*}
The same formula can be proved for any field extension $K/\bk$, hence we obtain
$$\sH^\Z(t)
=\prod_{\bd,r}\rbra{\sum_{n\ge0}\psi_r([c_n])t^{nr\bd}}^{[\sI_{r\bd,r}]}
=\prod_{\bd,r}\rbra{\psi_r\rbr{\sH_{\Vect}^\Z(t^\bd)}}^{[\sI_{r\bd,r}]}.
$$
Applying Theorem \ref{power} and Corollary \ref{cor:abs vs ind}, we obtain
$$\sH^\Z(t)=\prod_{\bd\in\Ga}\Pow(\sH_{\Vect}^\Z(t^{\bd}),[\sA_\bd]).$$
We will prove in Theorem \ref{th:vect spaces} that
$\sH_{\Vect}^\Z(t)=\Exp\rbr{\frac{[\Z]}{q-1}t}$.
Then
$$\sH^\Z(t)=\prod_{\bd\in\Ga}
\Exp\rbr{\frac{[\Z]}{q-1}[\sA_\bd]t^\bd}
=\Exp\rbr{\frac{[\Z]}{q-1}\sA(t)}.
$$
\end{proof}

\begin{corollary}
We have
\begin{enumerate}
\item $\sH_\cA^0(t)=\Exp\rbr{\frac{\sA_\cA(t)}{q-1}}$.
\item $\sH_\cA^*(t)=\Exp(\sA_\cA(t))$.
\item $\sH_\cA^a(t)=\Exp\rbr{\frac{q \sA_\cA(t)}{q-1}}$.
\item $\sH_\cA^Z(t)=\Pow(\cH_\cA^0(t),[Z])$, for every subvariety $Z\sbs\bA^1$.
\end{enumerate}
\end{corollary}

\begin{corollary}
\label{cor:compare ai}
Given a linear stack $\cA$, the volumes of absolutely indecomposable objects in $\cA^Z$ and in $\cA^0$ are related by the formula
\begin{equation}
\sA_{\cA^Z}(t)=[Z]\sA_{\cA^0}(t).
\end{equation}
\end{corollary}
\begin{proof}
Let $\cB=\cA^Z$ and $\cC=\cA^0$.
Then the stacks $\cB^0$ and $\cC^Z$ are equivalent.
Therefore
$$\sH_{\cB^0}(t)=\Exp\rbr{\frac1{q-1}\sA_{\cB}(t)}
=\sH_{\cC^Z}(t)
=\Exp\rbr{\frac{[Z]}{q-1}\sA_{\cC}(t)}.$$
This implies $\sA_\cB(t)=[Z]\sA_\cC(t)$.
\end{proof}

\begin{remark}
Given a quiver $Q$, let $\cA$ be the linear stack of its representations with $\cA(K)=\Rep(Q,K)$ -- the category of $Q$ representations over a field extension $K/\bk$.
The series
$$\sH^0(t)=\sum_{\bd\in\bN^{Q_0}}[\sH_\bd^0]t^\bd,\qquad
\sH_\bd^0(K)=\sum_{X\in\cA(K)}\frac{\#\End(X)^0}{\#\Aut(X)}t^{\udim X}$$
can be explicitly computed by fixing nilpotent endomorphisms (parametrized by partitions) at every vertex:
\begin{equation}
\sH^0(t)=\sum_{\la:Q_0\to\cP}\frac{\prod_{a:i\to j}q^{\ang{\la(i),\la(j)}}}
{\prod_{i}q^{\ang{\la(i),\la(i)}}(q\inv)_{\la(i)}}
\prod_{i\in Q_0}t_i^{\n{\la(i)}}
,
\end{equation}
where $\cP$ is the set of partitions
and
$(q)_\la=\prod_{i\ge1}(q)_{m_i}$,
$\ang{\la,\mu}=\sum_{i,j}m_in_j\min\set{i,j}=\sum_{i}\la'_i\mu'_i$,
for partitions $\la=(1^{m_1},2^{m_2},\dots)$ and $\mu=(1^{n_1}2^{n_2}\dots)$.
Then Theorem \ref{th:main proof} implies the Hua formula
\cite{hua_counting} (see \cite{mozgovoy_computational} for the plethystic formulation)
\begin{equation}
\sA(t)=(q-1)\Log(\sH^0(t)).
\end{equation}
\end{remark}

\subsection{The case of vector spaces}
Let us consider the linear stack $\cA=\Vect$ of vector spaces.
This means that $\cA(K)$ is the category of (finite-dimensional) vector space over $K$, for every field extension $K/\bk$.
Given $Z\sbs\bA^1_\bk$, consider the corresponding series of volumes
$$\sH^\Z(t)=\sH^\Z_{\Vect}(t)=\sum_{n\ge0}[\sH^\Z_n]t^n,\qquad
\sH^\Z_n(K)=\frac{\#M_n^\Z(K)}{\#\GL_n(K)}.$$

\begin{theorem}
\label{th:vect spaces}
We have
\begin{equation}
\sH^\Z_{\Vect}(t)=\Exp\rbr{\frac{[\Z]}{q-1}t}.
\end{equation}
\end{theorem}
\begin{proof}
Define $\sH^\Z(K;t)=\sum_{n\ge0}\sH_n^\Z(K)t^n\in\bQ\pser t$, for any field extension $K/\bk$.
The series $\sH^\Z(K;t)$ can be expressed as a weighted object count in the category $\cA^\Z(K)$ consisting of 
objects $X=(X_0,\vi)$, where $X_0$ is a vector space over $K$ and $\vi\in\End(X_0)^\Z$ (\cf \eqref{eq:HZ2})
$$\sH^\Z(K;t)
=\sum_{X\in\cA^\Z(K)\qt\sim}\frac{t^{\dim X}}{\#\Aut(X)}.$$
Indecomposable objects of the category $\cA^\Z(\bk)$ are of the form $X=\bk[x]/(f^i)$, where $f\in \bk[x]$ is monic irreducible with roots in $\Z(\ubar\bk)$ and $i\ge1$.
We have $\End(X)\iso \bk[x]/(f^i)$ and $\fs(X)=\bk[x]/(f)\iso\bF_{q^r}$, where $r=\deg f$.

Let $\cP$ be the set of maps $n:\bN^*=\bN_{>0}\to\bN$ with finite support -- any such map can be identified with the partition $\la=(1^{n_1}2^{n_2}\dots)$.
Given $n\in\cP$, consider modules
$$X=X_{f,n}=\bop_{i\ge1}(\bk[x]/(f^i))^{n_i},\qquad
Y=Y_{r,n}=\bop_{i\ge1}(K[x]/(x^i))^{n_i}
$$
over $\bk[x]$ and $K[x]$ respectively, where $K=\bF_{q^r}\iso\bk[x]/(f)$, $r=\deg f$.
Then $\#\Aut(X)=\#\Aut(Y)$.
Indeed, 
$$\dim_\bk\End(X)=r\sum_{i,j}n_in_j\min\set{i,j},\qquad
\dim_K\End(Y)=\sum_{i,j}n_in_j\min\set{i,j},$$
hence $\#\End(X)=\#\End(Y)$.
On the other hand $\fs(X)\iso\prod_i M_{n_i}(K)\iso\fs(Y)$.

Let $\Phi^\Z(K)$ be the set of monic irreducible polynomials in $K[x]$ with roots in $\Z(\ubar\bk)$, and let $\Phi^\Z_r(K)\sbs\Phi^\Z(K)$ be the subset of degree $r$ polynomials.
Every object of the category $\cA^\Z(\bk)$ is isomorphic to
$$X_\bn
=\bop_{f\in\Phi^\Z(\bk)} X_{f,\bn_f}
=\bop_{f\in\Phi^\Z(\bk),i\ge1}(\bk[x]/(f^i))^{\bn_{f,i}}
$$ 
for some map $\bn:\Phi^\Z(\bk)\to\cP$ with finite support.
We have $\Aut(X_\bn)\iso\prod_f\Aut (X_{f,\bn_f})$.
Therefore
\begin{multline*}
\sH^\Z(\bk;t)
=\sum_{\bn:\Phi^\Z(\bk)\to\cP}\prod_{f\in\Phi^\Z(\bk)}\frac{t^{\deg f\sum_i i\bn_{f,i}}}{\#\Aut (X_{f,\bn_f})}
=\prod_{f\in\Phi^\Z(\bk)}\rbr{\sum_{n\in\cP}\frac{t^{\deg f\sum_i in_i}}{\#\Aut(X_{f,n})}}\\
=\prod_{f\in\Phi^\Z(\bk)}\rbr{\sum_{n\in\cP}\frac{t^{\deg f\sum_i in_i}}{\#\Aut(Y_{\deg f,n})}}
=\prod_{f\in\Phi^\Z(\bk)} \sH^0(\bF_{q^{\deg f}};t^{\deg f})
=\prod_{r\ge1}\sH^0(\bF_{q^r};t^r)^{\#\Phi^\Z_r(\bF_q)},
\end{multline*}
where $\sH^0(K;t)=\sH^{\Z_0}(K;t)$, $\Z_0=\set0\sbs\bA^1$, corresponds to nilpotent endomorphisms.
The same formula can be proved for any field extension $K/\bk$, hence we obtain
$$\sH^\Z(t)=\prod_{r\ge1}\rbr{\psi_r \sH^0(t)}^{[\Phi^\Z_r]}
.$$
By Corollary \ref{cor:abs vs ind} (see also Remark \ref{rm:gauss}), 
we have 
$$\psi_n([\Z])=\sum_{r\mid n}r[\Phi^\Z_r].$$
Therefore, by Theorem \ref{power}, we obtain
$$\sH^\Z(t)
=\Pow(\sH^0(t),[\Z]).
$$
If $\Z=\Z_*=\bA^1\ms\set0$, we obtain
$$\Pow(\sH^0(t),q-1)=\sH^*(t)=\sum_{n\ge0} t^n=\Exp(t),$$
hence $\sH^0(t)=\Exp\rbr{\frac t{q-1}}$.
Therefore, in general, 
$\sH^\Z(t)
=\Exp\rbr{\frac {[\Z]}{q-1}t}$.
\end{proof}

\begin{remark}
Elements of $M_n(K)$ can be interpreted as finite-dimensional modules over $K[t]$, or equivalently, as zero-dimensional coherent sheaves over $\bA^1_K$.
Similarly, elements of $M_n^\Z(K)$ can be interpreted as zero-dimensional coherent sheaves over $\Z_K\sbs\bA^1_K$.
One can extend the above result to an arbitrary smooth curve $Z$ over \bk as follows.
Define the linear stack \cA with $\cA(K)=\Coh_0 Z_K$, the category of zero-dimensional coherent sheaves over $Z_K$,
for any field extension $K/\bk$.
Then the volume series of this stack is
\begin{equation}
\sH_\cA(t)=\Exp\rbr{\frac{[Z]}{q-1}t}.
\end{equation}
\end{remark}

\section{Commuting matrices}

\subsection{Linear stack of projective representations}
\label{sec:comm1}
Let $Q$ be an acyclic quiver and let $\bk=\bF_q$ be a finite field.
Define the linear stack \cA of projective $Q$-representations as follows.
For any field extension $K/\bk$, define $\cA(K)$ to be the category of projective $Q$-representations over $K$
and, for any field extension $L/K/\bk$, define the restriction functor
$$\rho_{L/K}:\cA(K)\to\cA(L),\qquad P\mto L\ts_KP.$$
There is an isomorphism $K_0(\cA(K))\iso\Ga=\bZ^{Q_0}$ constructed as follows.
Let $A=KQ$ be the path algebra of $Q$ over the field $K$.
We will identify the category of $Q$-representations over $K$ with the category of left $A$-modules.
For every vertex $i\in Q_0$, let $e_i\in A$ be the idempotent given by the trivial path at the vertex $i$.
Then $P(i)=Ae_i$ is an indecomposable projective $A$-module
and every projective $A$-module $P$ is isomorphic to $P_\bd=\bop_{i\in Q_0}P(i)^{d_i}$, for a unique $\bd\in\bN^{Q_0}$. 
We define $[P:P(i)]=d_i$ and define an isomorphism
$$\cls:K_0(\cA(K))\to\Ga=\bZ^{Q_0},\qquad P\mto([P:P(i)])_{i\in Q_0}.$$
For any $\z_1,\z_2\in\set{0,*,a}$, define the volume series $\sH^{\z_1,\z_2}(t)$
as in the introduction \eqref{eq:Hss1}:
\begin{equation}
\sH^{\z_1,\z_2}(t)=\sum_{\bd\in\bN^{Q_0}}\frac{\vl{ C^{\z_1,\z_2}_\bd}}{\vl{\Aut(P_\bd)}}t^\bd\in\cV\pser{t_i\rcol i\in Q_0}.
\end{equation}
where $C^{\z_1,\z_2}_\bd\sbs A_\bd^{s_1}\xx A_\bd^{s_2}$
is the subvariety of commuting pairs and $A_\bd=\End_A(P_\bd)$.
It is convenient to reformulate this definition using the above linear stack.
For every $\z\in\set{0,*,a}$, define the subvariety $\Z_\z\sbs \bA^1$ as in \eqref{eq:Zs2}
and define the linear stack $\cA^s=\cA^{Z_s}$ as in
\S\ref{sec:counting endo}.
Define the linear stack $\cA^{s_1,s_2}=(\cA^{s_1})^{s_2}
\iso(\cA^{s_2})^{s_1}$.
Then (\cf \eqref{eq:vol series1})
\begin{equation}
\sH^{s_1,s_2}(t)
=\sH_{\cA^{s_1,s_2}}(t)
=\sum_\bd[\sH_{\bd}]t^\bd,\qquad
\sH_{\bd}(K)
=\sum_{\ov{X\in\cA^{s_1,s_2}(K)\qt\sim}{\cls X=\bd}}\frac1{\#\Aut(X)}.
\end{equation}
We are ready to prove Theorem \ref{th:main1} from Introduction.

\begin{theorem}
\label{th:main1 proof}
For any $\z_1,\z_2\in\set{0,*,a}$, we have
$$\sH^{\z_1,\z_2}(t)=\Exp\rbr{\frac{[\Z_{\z_1}] [\Z_{\z_2}]}{q-1}\sA(t)},$$
where $\sA(t)=\sA_{\cA^0}(t)$ is the volume series of absolutely indecomposable objects of $\cA^0$
\eqref{eq:ai volume1}.
\end{theorem}
\begin{proof}
We have by Theorem \ref{th:main proof}
$$\sH^{s_1,s_2}(t)
=\sH^{s_2}_{\cA^{s_1}}(t)
=\Exp\rbr{\frac{[Z_{s_2}]}{q-1}\sA_{\cA^{s_1}}(t)}.
$$
On the other hand $\sA_{\cA^{s_1}}(t)=[Z_{s_1}]\sA_{\cA^0}(t)=[Z_{s_1}]\sA(t)$ by Corollary \ref{cor:compare ai}. 
\end{proof}

\begin{corollary}
We have 
$$H^{*,0}(t)=\Pow(H^{0,0}(t),q-1).$$
\end{corollary}


\begin{remark}
\label{rem: s method}
Given a linear stack \cB, define, as before, the linear stack $\cB^0$ consisting of pairs $(X,\vi)$, where $X$ is an object from $\cB$ and $\vi\in\End(X)$ is nilpotent.
There is a general method of counting (weighted) volumes of $\cB^0$ due to Schiffmann \cite{schiffmann_indecomposable,mozgovoy_countinga}.
It reduces the object counting in $\cB^0$ to the counting of certain flags in $\cB$.
However, a thorough analysis of this method \cite{mozgovoy_countinga} shows that one requires \cB to
be abelian and hereditary (meaning that the corresponding categories of \cB are abelian and hereditary).
In our case we need to consider $\cB=\cA^0$, where $\cA$ is the linear stack of projective quiver representations.
The stacks \cA and \cB are not abelian,
but we can embed $\cA\sbs\bar\cA$ into the stack of all quiver representations and then embed $\cB\sbs\bar\cB=\bar\cA^0$.
The stack $\bar\cB$ has homological dimension two, but the second \Ext-groups between the objects of \cB vanish, hence we can say that \cB is rather close to a hereditary category.
Nevertheless, this technical difficulty seems to be a serious obstacle for the application of the above mentioned method.
 \end{remark}

\begin{remark}
\label{rem:evs result}
It was proved by Evseev \cite{evseev_conjugacy,evseev_groups} that the original Higman's conjecture (polynomiality for the quiver of type $A_n$ and vector $\bd=(1^n)$ in our language)
would follow from polynomiality of the commuting varieties $C_\bd^{*,0}$ for the quiver $Q=[1\to 2]$ and arbitrary $\bd\in\bN^2$.
Projective representations of this quiver correspond to injective linear maps $f:V_1\to V_2$ between vector spaces.
By the previous results, polynomiality of $C_\bd^{*,0}$ is equivalent to polynomiality of (unweighted) volumes of absolutely indecomposable objects in $\cA^0$, where $\cA$ is the stack of projective $Q$-representations,
as well as polynomiality of unweighted volumes of $\cA^0$.
An object in $\cA^0(\bk)$ can be specified by nilpotent representations of the Jordan quiver $C_1$ (having one vertex and one loop) at every vertex of $Q$ and an injective homomorphism between them.
These nilpotent representations have the form $V_1\iso I_\la=\bop_{i\ge1}\bk[x]/(x^{\la_i})$ and $V_2\iso I_\nu$, for some partitions $\la,\nu$.
Then the number of isomorphism classes of object in $\cA^0(\bk)$ having class $\bd$ is
$$\rho_{(d_1,d_2)}=\sum_{\n\la=d_1,\n\nu=d_1+d_2}\#\Hom^\inj(I_\la,I_\nu)/(A_\la\xx A_\nu),$$
where $A_\la=\Aut(I_\la)$.
One can conjecture that every summand here is polynomial-count.
One can also refine this even further by considering the set $\cP_{\mu\la}^\nu$ of short exact sequences
$$0\to I_\la\to I_\nu\to I_\mu\to0,$$
equipped with an action of $A_\la\xx A_\nu\xx A_\mu$, and observing that
$$\#\Hom^\inj(I_\la,I_\nu)/(A_\la\xx A_\nu)
=\sum_\mu\#\cP_{\mu\la}^\nu/(A_\la\xx A_\nu\xx A_\mu).$$
Then one can ask if every summand here is polynomial-count.
Note that $\#\cP_{\mu\la}^\nu/(A_\la\xx A_\mu)$ are polynomial-count as they are just the classical Hall polynomials \mac.
Note also that
$$\bigsqcup_\nu\cP_{\mu\la}^\nu/A_\nu\iso\Ext^1(I_\mu,I_\la)
\iso\Hom(I_\la,I_\mu)^*,
$$
hence 
$$\rho_{(d_1,d_2)}=\sum_{\n\la=d_1,\n\mu=d_2}\#\Hom(I_\la,I_\mu)/(A_\la\xx A_\mu)$$
and one can again conjecture that every summand is polynomial-count.
Note that the last sum corresponds to an unweighted volume of the stack $\cB^0$, where $\cB$ is the stack of all $Q$-representations (not just projective).
\end{remark}

\subsection{Examples}
\label{sec:examples}
Consider the quiver of type $A_n$:
$[1\to2\to\dots\to n]$
and vector $\bd=(1^n)=(1,\dots,1)$.
We have $B_n=\Aut(P_\bd)$ and $U_n=1+\End(P_\bd)^0$.
Therefore
\begin{gather}
\sH_\bd^{0,0}(\bF_q)
=\frac{\#\sets{(x,y)\in U_n\xx U_n}{xy=yx}}{\#B_n}
=\frac1{(q-1)^n}\ga(U_n,U_n).\\
\sH_\bd^{*,0}(\bF_q)
=\frac{\#\sets{(x,y)\in B_n\xx U_n}{xy=yx}}{\#B_n}
=\ga(B_n,U_n).
\end{gather}

Let $\cA$ be the linear stack of projective quiver representations and let $\cA^0$ be the linear stack of pairs $(P,\vi)$, where $P$ is a projective quiver representation and $\vi\in\End(P)$ is nilpotent.
As before, we define $\sA(t)=\sA_{\cA^0}(t)=\sum_\bd[\sA_\bd]t^\bd$, where $[\sA_\bd]$ counts isomorphism classes of absolutely indecomposable objects in $\cA^0$ having class $\bd$.
Then Theorem \ref{th:main1 proof} implies
\begin{equation}
\label{eq:H00}
\sH^{0,0}(t)
=\sum_\bd [\sH^{0,0}]t^\bd
=\Exp\rbr{\frac{\sA(t)}{q-1}},\qquad
\sH^{*,0}(t)
=\sum_\bd [\sH^{*,0}]t^\bd
=\Exp(\sA(t)).
\end{equation}

\begin{remark}
Note that the above formulas imply that we can express $\sH^{0,0}_{(1^n)}$ and $\sH^{*,0}_{(1^n)}$ in terms of
$\sA_{(1^k)}$ for $k\le n$ (and vice versa).
An explicit formula relating $\sH^{*,0}_{(1^n)}$ and 
$\sH^{0,0}_{(1^k)}$ can be found in \cite[\S4.2]{evseev_groups}.
One can also derive it from the above equations.
\end{remark}

Invariants $\al_n=\ga(U_n,U_n)$ were computed in \cite{vera-lopez_conjugacya,
vera-lopez_some,
vera-lopez_conjugacy}
for $n\le 13$ and in \cite{pak_higmans} for $n\le 16$.
Here is the list for $n\le10$:
\begin{flalign*}
\al_1&=1&\\
\al_2&=q\\
\al_3&=q^{2}+q-1\\
\al_4&=2q^{3}+q^{2}-2q\\
\al_5&=5q^{4}-5q^{2}+1\\
\al_6&=q^{6}+12q^{5}-5q^{4}-15q^{3}+5q^{2}+4q-1\\
\al_7&=8q^{7}+28q^{6}-35q^{5}-35q^{4}+35q^{3}+7q^{2}-7q
\\
\al_8&=4q^{9}+38q^{8}+48q^{7}-168q^{6}-28q^{5}+161q^{4}-28q^{3}-32q^{2}+4q+2
\\
\al_9&=3q^{11}+39q^{10}+146q^{9}-75q^{8}-606q^{7}+364q^{6}+504q^{5}-381q^{4}-53q^{3}+57q^{2}+6q-3
\\
\al_{10}&=5q^{13}+45q^{12}+240q^{11}+322q^{10}-1255q^{9}-1185q^{8}+2880q^{7}+310q^{6}-2124q^{5}\\
&+565q^{4}+280q^{3}-60q^{2}-25q+3
\end{flalign*}

Applying Theorem \ref{th:main1 proof} we can determine invariants $\sA_n=[\sA_{(1^n)}]$ from the above invariants:
\begin{flalign*}
\sA_1&=1&\\
\sA_2&=1\\
\sA_3&=1\\
\sA_4&=2\\
\sA_5&=5\\
\sA_6&=q+17&\\
\sA_7&=8\,q+69\\
\sA_8&=4\,{q}^{2}+66\,q+334&\\
\sA_9&=3\,{q}^{3}+63\,{q}^{2}+530\,q+1855\\
\sA_{10}&=5\,{q}^{4}+90\,{q}^{3}+840\,{q}^{2}+4492\,q+11673
\end{flalign*}
Similarly, one can determine invariants $\ga(B_n,U_n)$
from the invariants $\ga(U_n,U_n)$ for $n\le16$ \cite[\S4.2]{evseev_groups}.
One can observe that invariants $\sA_n$ are significantly simpler than invariants $\ga(U_n,U_n)$ and somewhat simpler than invariants $\ga(B_n,U_n)$.
This can be explained by equations \eqref{eq:H00}.
Note that we have $\sA_n\in\bN[q]$ in all known examples, as one would expect from analogues of Kac polynomials (or from Donaldson-Thomas invariants) for the stack $\cA^0$.

\providecommand{\bysame}{\leavevmode\hbox to3em{\hrulefill}\thinspace}
\providecommand{\href}[2]{#2}

\end{document}